\documentclass[12pt]{amsart}
\usepackage{amsmath,amssymb,amsthm,mathtools}
\usepackage{fourier}
\usepackage{tikz}
\usepackage{graphicx}
\usepackage{subfig}

\theoremstyle{plain}
\newtheorem{theorem}{Theorem}
\newtheorem{lemma}[theorem]{Lemma}
\newtheorem*{theorem*}{Theorem}

\newcommand{\N}{\mathbb{N}}

\begin{document}
\author{Iwan Praton}
\author{Nart Shalqini}
\address{Franklin \& Marshall College} \email{ipraton@fandm.edu,
nshalqin@fandm.edu}
\title{Amicable Heron Triangles}
\date{}
\maketitle

\begin{abstract}
A Heron triangle is a triangle whose side lengths
and area are integers. Two Heron triangles are
\emph{amicable} if the perimeter of one is the area of
the other. We show, using elementary techniques,
that there is only one pair of amicable Heron triangles.
\end{abstract}

\section*{Introduction}
A Heron triangle is a triangle whose side lengths
and area are integers. They are named for the
Greek mathematician Heron (or Hero) of Alexandria,
who is usually credited with inventing the
formula for the area $A$ of a triangle in terms of its
side lengths $a, b, c$: 
\[
A=\sqrt{s(s-a)(s-b)(s-c)};
\]
here $s$ is the semiperimeter $\frac12(a+b+c)$. 

Heron triangles form a popular topic (e.g., \cite{Ca},
\cite{Ch}, \cite{N}, \cite{Y}),
and new facts about them are still being discovered.
For example, Hirakawa and Matsumura \cite{KM} showed
recently that there is a unique pair (up to scaling)
of right and isosceles Heron triangles with
the same perimeter and the same area. Surprisingly,
the proof
uses sophisticated tools from the theory of
hyperelliptic curves. The result has been featured
in a Numberphile video \cite{H}.

The video focuses mostly on \emph{equable}
Heron triangles (called Super-Hero triangles
in the video), i.e., Heron triangles where the
perimeter is equal to the area. This seems
analogous to \emph{perfect numbers}. Recall
that a perfect number is a
positive integer whose aliquot sum
is equal to itself. (The aliquot sum of $n$
is the sum of the divisors of $n$, excluding
$n$.) In
equable Heron triangles, the perimeter and
area play the roles of $n$ and its aliquot sum.

Venturing
beyond a single number, recall that a pair
of positive integers $n$ and $m$ form an
\emph{amicable pair} if the aliquot sum of $n$
is equal to $m$ and the aliquot sum of $m$
is equal to $n$. 
Analogously,
we define two Heron triangles $H_1$ and
$H_2$ to be \emph{amicable} if the area
of $H_1$ is equal to the perimeter of $H_2$
and the perimeter of $H_1$ is equal to
the area of $H_2$. 

Amicable Heron triangles exist:
the triangles with side lengths $(3, 25, 26)$
and  $(9, 12, 15)$ form an example. They
are an unusual looking pair. 
(See Figure~\ref{fig:AB}.)

\begin{figure}[htbp]
  \centering
  \subfloat{\includegraphics[width=1in]{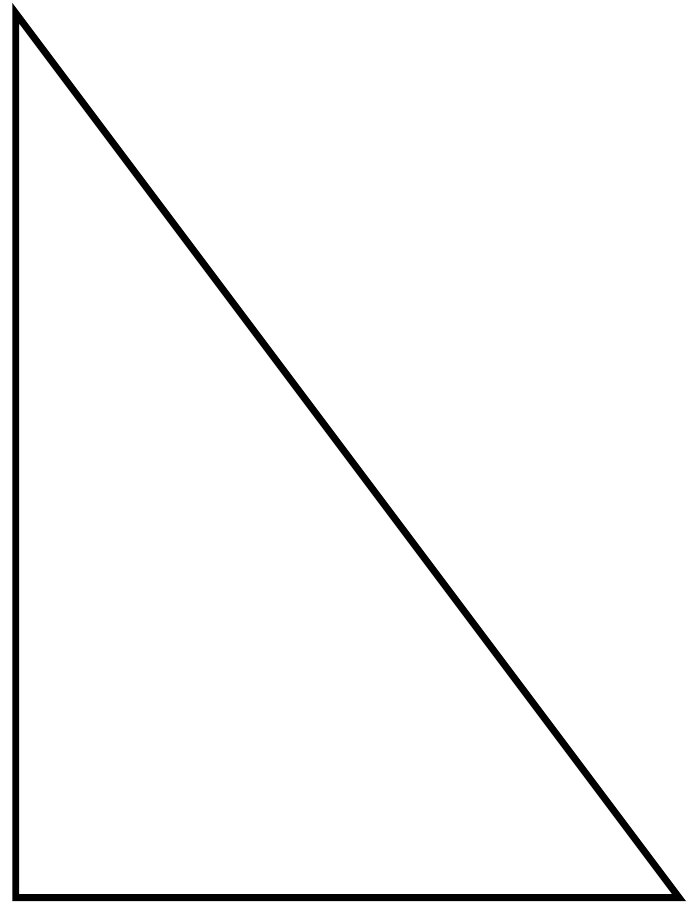} \label{fig:b}}\\
  \subfloat{\includegraphics[width=4in]{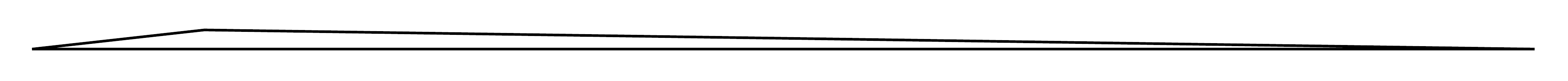} \label{fig:a}} 
  \caption{A unique pair of triangles} \label{fig:AB}
\end{figure}

Somewhat surprisingly, there is no other example.

\begin{theorem*}
There is only one pair of amicable Heron
triangles: the $(3,25,26)$ and $(9,12,15)$
triangles.
\end{theorem*}

In contrast to \cite{KM}, we use completely elementary
methods to prove this result. 

\section*{Proofs}
We begin by establishing our notation.
Suppose we have a Heron triangle with side lengths 
$a,b,c$. Its semiperimeter $s=(a+b+c)/2$ is an
integer. We define $x=s-a$, $y=s-b$,
$z=s-c$, and set $x\leq y\leq z$.
Then Heron's formula for the area
of the triangle becomes $\sqrt{sxyz}$. 

\begin{lemma}
Suppose $H$ is one of a pair of amicable Heron
triangles with perimeter $p$ and area $A$.
Then $A$ divides $2p^2$.
\end{lemma}
\begin{proof}
For any Heron triangle, 
\[
\frac{2A^2}{p} = \frac{2sxyz}{2s}=xyz
\]
is an integer. Apply this result to the partner
triangle of $H$. Since the perimeter of $H$
is equal to the area of its partner and the area
of $H$ is equal to the perimeter of its partner,
we get the result of the lemma.
\end{proof}

We now take care of the case where
both amicable triangles are equable. It is well-known
that there are only five equable Heron triangles \cite{D}:
triangles with side lengths 
$(5,12,13), (6,8,10), (6,25,29), (7,15,20)$, and $(9,10,17)$.
Their perimeters (and thus their areas) are all different, so
none of them form an amicable pair. Equable
triangles are not amicable.

We conclude that if we have an amicable pair, 
then for one of the triangles the perimeter 
is larger than the area. These triangles are long
and skinny, similar to  the second triangle in Figure~\ref{fig:AB}.
There are not many such triangles.
From now on, let $H$ denote a triangle
of this kind. 

\begin{lemma}
For $H$ as above, 
\[
4(x+y+z) > xyz. \eqno(*)
\]
\end{lemma}
\begin{proof}
This is a simple consequence of the perimeter
of $H$ being larger than the area of $H$. 
\end{proof}
These two lemmas are our main tool for cutting
down the possible values of $x,y,$ and $z$.
In fact, they suffice to show that there are only
finitely many.

\begin{lemma}
Let $H$ be as above. Then there are only a finite
number of $x, y, z$ values that satisfy lemmas
1 and 2.
\end{lemma}
\begin{proof}
We first show that $x\leq 3$. If $x\geq 4$, then by
Lemma 2, 
\[
4(z+z+z)\geq 4(x+y+z)>xyz\geq 4\cdot 4\cdot z = 16z,
\]
a contradiction, so $x\leq 3$.

Now we show that $y\leq 9$. If $y\geq 10$, then
\[
4(3+z+z)\geq 4(x+y+z)>xyz\geq 10z \implies 12 + 8z > 10z
\implies 6>z,
\]
which is a contradiction since $z\geq y\geq 9$. We
conclude that there are only finitely many values of $x$
and $y$.

We now tackle $z$. Lemma 1 states that $2p^2/A$ is
an integer, which means
\[
\frac{8s^2}{\sqrt{sxyz}}\in \N
\implies \frac{64s^4}{sxyz}\in \N
\implies \frac{64(x+y+z)^3}{xyz}\in\N
\implies \frac{64(x+y+z)^3}{z}\in \N.
\]
Let $c=x+y$. Then $64(z+c)^3/z=
64(z^2+3zc+3c^2+c^3/z)$ is an integer,
which implies that $64c^3/z$ is an integer.
So $z$ must be a divisor of $64(x+y)^3$, and
since there are only finitely many values of 
$64(x+y)^3$, there are only finitely many values
of $z$ also.
\end{proof}

We now need to investigate only a finite number
of cases. It turns out that the possibilities  can
be cut down considerably by using the requirement
that the area of $H$ is an integer. We provide
an example here; other cases are similar.

Suppose $x=1$ and $y=4$. The $2p^2/A$
calculation in Lemma 3 shows that $4z$ divides
$64(z+5)^3$;
we conclude that $z$ is a divisor of $2^4\cdot 5^3$.
Since $z\geq 4$, there are $18$ possibilities
for $z$. The area of $H$ is $\sqrt{4z(z+5)}$;
among the 18 possible value of $z$, only
$z=4$ produces an integer area. Thus this
case produces only one possibility: $x=1$,
$y=4$, and $z=4$.

Indeed, when we check all possible values of
$x,y,z$, we come up with just four cases that
satisfy all the requirements mentioned above:
\begin{itemize}
\item
$x=1, y=4, z=4$, producing $H$ with side lengths $5,5,8$.
\item
$x=1, y=2, z=3$, producing $H$ with side lengths $3,4,5$.
\item
$x=1, y=2, z=24$, producing $H$ with side lengths $3,25,26$.
\item
$x=1, y=2, z=864$, producing $H$ with side lengths $3,865,866$.
\end{itemize}

The first two cases are easy to eliminate. The first case produces a triangle of area $12$ and perimeter $18$. Its amicable partner, if it exists, must have semiperimeter $6$. This yields three possibilities: $x=1, y=2, z=3$ or $x=1, y=1, z=4$ or $x=2, y=2, z=2$, none of which yields an area of $18$. The second case produces a triangle of area $6$, but it is impossible to have a partner triangle of perimeter $6$, the only possibility being an equilateral triangle with side-length $2$ that has an irrational area.

If $H$ is the fourth triangle listed above then its perimeter is 1734 and its area is 1224.
Thus its partner triangle has perimeter 1224 and
area 1734. Therefore for the partner triangle,
we have $xyz=1734^2/612=4913$, an odd number,
which implies that $x,y,$ and $z$ are all odd. 
This contradicts that the semiperimeter of the
partner triangle is 612. 
Therefore $H$ has no partner triangle. 

The third case produces the amicable pair mentioned
in the Theorem. 
This concludes the proof of the Theorem: there
is a unique pair of amicable Heron triangles.

\end{document}